\documentclass[a4paper,10pt,twoside]{article}
\usepackage[utf8]{inputenc} 
\usepackage{graphicx}
\usepackage[pdftex,bookmarks]{hyperref}
\usepackage{amsmath}    
\usepackage{amssymb}
\usepackage{amsmath,amsthm}
\usepackage{textcomp}
\usepackage{color}
\usepackage[all]{xy}

\usepackage[T1]{fontenc}
\usepackage{sansmath}
\sansmath

\numberwithin{equation}{section}

\theoremstyle{plain}

\newtheorem{teo}{Theorem}[section]
\newtheorem{lem}[teo]{Lemma}
\newtheorem{propo}[teo]{Proposition}

\newtheorem{fact}[teo]{Fact}

\newtheorem{conjecture}[teo]{Conjecture}
\newtheorem{problem}[teo]{Problem}

\theoremstyle{definition}

\newtheorem{ddef}[teo]{Definition}

\newtheorem{rem}[teo]{Remark}
\newtheorem{claim}[teo]{Claim}

\theoremstyle{remark}

\newcommand{\fistar}{\varphi^*}
\newcommand{\ffi}{\varphi}
\newcommand{\OO}{\mathcal{O}}
\newcommand{\LL}{\mathcal{L}}

\newcommand{\PP}{\mathbb{P}}
\newcommand{\ZZ}{\mathbb{Z}}

\newcommand{\CC}{\mathbb{C}}
\newcommand{\RR}{\mathbb{R}}

\def\Pos{\overline{\operatorname{Pos}}}
\def\PosO{\operatorname{Pos}}

\def\Amp{\operatorname{Amp}}

\def\Neg{\operatorname{Neg}}

\def\Nef{\operatorname{Nef}}
\def\Eff{\operatorname{Eff}}

\def\Pic{\operatorname{Pic}}

\def\NE{\overline{\operatorname{NE}}}

\def\Bl{\operatorname{Bl}}
\def\cl{\operatorname{cl}}
\def\mult{\operatorname{mult}}

\def\intern{\operatorname{int}}

\newcommand{\cc}{\cdot}

\definecolor{sap}{RGB}{129,36,51}

\usepackage{titlesec}

\titleformat{\section}[hang]
{\normalfont \LARGE \color{sap}}
{\thesection \phantom{M}}
{0pt}
{}[]

\titleformat{\subsection}[hang]
{\normalfont \Large \color{sap}}
{\thesubsection \phantom{M}}
{0pt}
{}[]

\usepackage{geometry}
\usepackage{fancyhdr}
\fancypagestyle{plain}{
	\fancyhead{} 
	\fancyfoot{} 

}

\title{{\color{sap}On the influence of the Segre Problem\\on the Mori cone of blown-up surfaces}}
\author{Fulvio Di Sciullo}

\begin{document}
\maketitle
%
%
%

\begin{abstract}
We propose a generalization of SHGH Conjectures to a smooth projective surface Y: the so called Segre Problem. The study of linear systems on Y can be translated in terms of the Mori cone of the blow up $X=\Bl_r Y$ at $r$ general points. Generalizing a result from \cite{deF}, we prove that if Segre Problem holds true, then a part of $\NE(X)$ does coincide with a part of the positive cone of $X$.
\end{abstract}


%
%

%


\section*{Introduction}

As it is well-known, one of the first goals achieved by Mori Theory was a description of the $K_X$-negative part of the Mori cone $\NE(X)$ of a projective variety; in this paper we deal with the structure of the $K_X$-positive part of this cone in the case of blown-up surfaces.


The celebrated Nagata Conjecture on linear system on $\PP^2$, in the spirit of conjectures of Segre, Harbourne, Gimigliano and Hirshowitz (SHGH conjectures), is strictly related to the shape of the Mori cone of $X= \Bl_r \PP^2$, the blowing up of the plane at $r$ general points. A consequence of SHGH conjectures is the decomposition
$$
\NE(X)= \Pos(X) + \sum R(C),
$$
where $\Pos(X)$ is the positive cone and the sum runs on $(-1)$-curves.

In order to generalize this kind of conjectures to any blown-up surface, we get interested in integral curves with negative self-intersection. We focus on a smooth projective surface $Y$ and we transfer the study of linear systems of curves on $Y$ passing through $r$ general points $x_1, \ldots, x_r$ with some multiplicities, to the study of curves on  $X =\Bl_r Y$ with negative self-intersection. 

We ask ourselves the natural generalized reformulation of SHGH conjectures:
%



\begin{quote}
	\textbf{Problem.} \emph{Let $X= \Bl_r Y$ a blown-up surface at $r$ general points; let us suppose $h^2(X,L)=0$ for all line bundles $L$ associated to a non exceptional and non empty linear system $\LL$. If moreover $\LL$ is reduced, then $\LL$ is non special.}
\end{quote}
This can easily seen to be false in a number of situations (see Section \ref{sec:contro}); the so called Segre Problem (Problem \ref{conj:segre}) is the refined statement of that problem. 

Since a consequence of the Segre Problem is the boundedness of negativity and arithmetic genus for the curves with negative self-intersection, we get to the statement of our main result: if the problem has a positive solution holds true, then a part of $\NE(X)$ is circular.

%
%



\begin{quote}
\textbf{Main Theorem.}\emph{ Let $X=\Bl_r Y$ the blow up at $r$ general points of a smooth projective surface $Y$ and let $L$ be the pullback to $X$ of an ample $A$ on $Y$. Let us suppose that for every integral curve $C \subset X$ with negative self-intersection, $C^2\geqslant -\nu_X$ and $p_a(C)\leqslant \pi_X$.}

\emph{If $r$ is large enough (explicit bounds depending only on $A$, $\nu_X$ and $\pi_X$), then there exists an explicit $s \in \RR$ such that}
$$
\NE(X)_{(K-sL)^{\geqslant 0}}= \Pos(X)_{(K-sL)^{\geqslant 0}}.
$$
\emph{In particular this is verified if $r \gg 0$ and the Segre Problem has solution.}
\end{quote}

The content of the theorem is pictured, in the case $\rho(X)=3$, in Figure \ref{fig:intro}.
\begin{figure}[!h]
  \centering
    \includegraphics[scale=0.55]{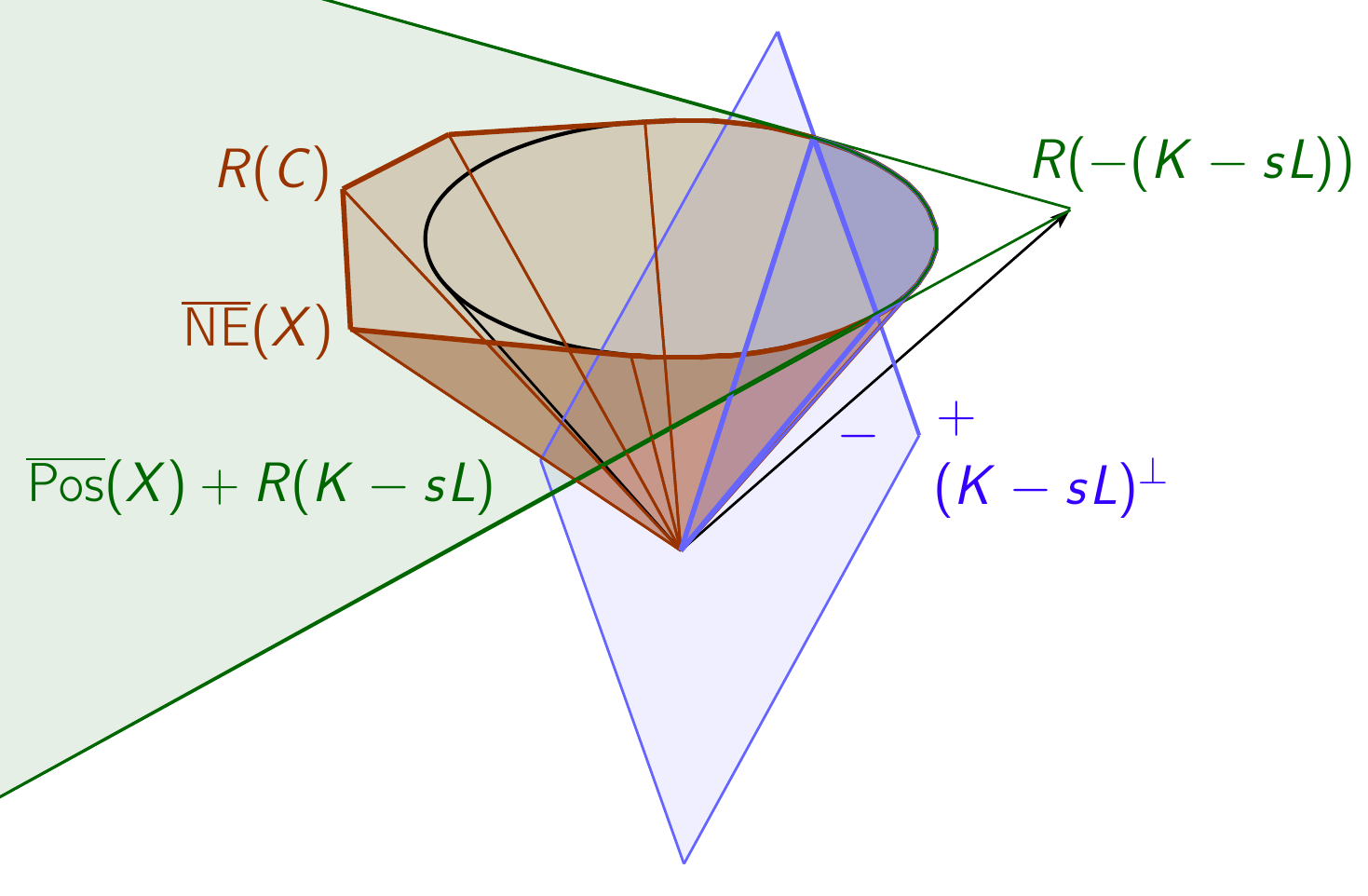}
  \caption{The $(K-sL)$-positive part of $\NE(X)$ in the $\rho(X)=3$ case.}
 \label{fig:intro}
\end{figure}

Finally we show that our result is, in some sense, sharp; it is not possible to work with $K^{\perp}$ in the statement of the Main Theorem: we must consider $(K-sL)^{\perp}$. We prove that, independently from any Conjecture, in many meaningful examples we have:
$$
\Pos(X)_{{K_X}^{\geqslant 0}} \subsetneq \NE(X)_{{K_X}^{\geqslant 0}}.
$$
More precisely, we know that this happens if the blown-up surface $Y$ is not uniruled and $r$ is sufficiently large (see Proposition \ref{prop:uni}) or if the inequalities of Proposition \ref{propoAlfa} are verified.

\paragraph{Notations.}For standard definitions about positivity topics and Mori Theory we refer to the classical books by Lazarsfeld (\cite{laz2}), Debarre (\cite{deb}) and Koll{\'a}r and Mori (\cite{km}). Throughout this paper we will work over the field $\CC$ of complex numbers.
\paragraph{Acknowledgements.}The author is deeply grateful to his advisor Prof. Angelo Felice Lopez for many enlightening comments and encouraging suggestions. He wishes also to thank Prof. Tommaso de Fernex and Prof. Ciro Ciliberto for useful discussions. He wants also to mention Prof. Andreas Leopold Knutsen, Dr. Salvatore Cacciola and Dr. Lorenzo Di Biagio for answering many questions.

\section{Concerning cones on surfaces}\label{sec:cones}

Throughout this section $S$ will denote a smooth projective surface; in particular, we will be mainly interested in the study of curves with negative self-intersection. We fix the notation with the following definition.



\begin{ddef}[$(-n,p)$-curves]\label{ddef:npcurves}
An integral curve $C$ on a smooth surface $S$ is said to be a \emph{$(-n,p)$-curve} if $C^2 = -n$ and it has arithmetic genus $p_a(C)=p$. In particular a $(-n,0)$-curve is a $(-n)$-curve. A ray $R(C)$ in $\NE(S)$ is a \emph{$(-n,p)$-ray} if $R(C)$ is generated by a $(-n,p)$-curve $C\subset S$.
\end{ddef}

Since we are dealing with surfaces, we have $N_1(S) = N^1(S)$ and we shall denote it $N(S)$; in this space it is possible to compare cones spanned by classes of curves and by classes of divisors. Namely we would like to study the relationship between $\Nef(S)$ and $\NE(S)$; to this end it is useful to introduce an other cone.

\begin{ddef}[Positive cone]
Let $S$ be a smooth projective surface and let $h \in \Amp(S)$. The \emph{open positive cone of $S$} is $\PosO(S) = \left\{ x \in N^1(S) \mid x^2 > 0, x \cdot h > 0 \right\}$. The \emph{positive cone of $S$} is
\begin{equation}
\Pos(S) = \left\{ x \in N^1(S) \mid x^2 \geqslant 0, x \cdot h \geqslant 0 \right\}.
\end{equation}
\end{ddef}
The $\rho$-dimensional vector space $N(S)$ can be equipped with the Euclidean topology and by Hodge Index Theorem (see \cite[Theorem V.1.9]{hart}), the intersection form is a bilinear form on $N(S)$ with signature $(1, \rho-1)$ and the Sylvester theorem assures us the existence, for an ample class $h$, of a basis $\{e_1, \ldots, e_{\rho}\}$ such that
\begin{equation}\label{eq:base}
\begin{array}{ll}
e_1 = \frac{h}{\sqrt{h^2}} & {e_1}^2 = 1\\
{e_i}^2=-1 & \text{for }i=2, \ldots, \rho \\
e_i \cc e_j = 0 & \text{for }1 \leqslant i<j\leqslant \rho.
\end{array}
\end{equation}
Hence the intersection matrix is $\text{diag}(1,-1,\ldots, -1)$; we will use this basis to write the elements $x \in N(S)$ as $x = \sum_{i=1}^{\rho} x_i e_i$.

To visualize cones in $N(S)$ it could be useful to consider a slice of the cone with an hyperplane far from the origin; to this end we can fix the hyperplane $\Pi = (x_1 =1)$.

It is immediate to see that with the choices of (\ref{eq:base}), the positive cone $\Pos(S)$ has the following equations:
\begin{equation}
\Pos(S)= \left\{x \in N(S) \mid x_1 \geqslant 0, {x_1}^2 \geqslant \sum_{i=2}^{\rho}{x_i}^2 \right\}.
\end{equation}

\begin{fact}\label{factPos(iii)}
If $x, y \in \Pos(S)$, then $x \cc y \geqslant 0$; moreover if $x \neq 0$ and $y \in \PosO(S)$ or $y \neq 0$ and $x \in \PosO(S)$, we have that $x \cc y >0$. In particular, the positive cone $\Pos(S)$ is a convex cone.
\end{fact}
As immediate consequence of the former fact (for the proof, see \cite{bpvdv}), we have that the definition of $\Pos(S)$ does not depend on the choice of the ample class $h$. 

%


The following easy Lemma essentially gives a visual way to find the orthogonal hyperplane in $N(S)$ corresponding to a class $\gamma$: if $\gamma$ is outside $\Pos(S)$, $\gamma^{\perp}$ is simply the hyperplane passing through the intersection points of $\partial \Pos(S)$ with the tangent lines to $\Pos(S)$ coming out from $\gamma$.

\begin{lem}\label{lemmaConoMio}
Let $\gamma$ be a class in $N(S)$, with $\rho(S)\geqslant 3$, such that $\gamma^2 <0$, $\gamma \cc h \geqslant 0$ and let us consider $0 \neq \alpha \in \Pos(S)$; let $L$ be the line joining $\alpha$ to $\gamma$, then
$$
L \cap \Pos(S) =\{\alpha\} \quad \iff \quad \alpha^2 = \alpha \cc \gamma = 0.
$$
\end{lem}

Let us see what happens when $\gamma^2\geqslant 0$.

\begin{lem}\label{iperpianiC^2>0}
Let $S$ be a smooth projective surface with $\rho(S)\geqslant 2$ and let $0 \neq \gamma \in N(S)$ be a class with $\gamma \cc h \geqslant 0$ and $\gamma^2 \geqslant 0$.
\begin{enumerate}
\item If $\gamma^2 >0$, then $\gamma^{\perp} \cap \Pos(S)= \{0\}$;
\item if $\gamma^2=0$, then $\gamma^{\perp} \cap \Pos(S)= R(\gamma)$.
\end{enumerate}
\end{lem}

In the spirit of comparing cones, we have the following well-known properties.
\begin{fact}\label{factPos(vi)}
If $S$ is a projective smooth surface, then
\begin{enumerate}
\item $\Pos(S)= \left( \Pos(S) \right)^{\vee}$;
\item $\Nef(S) \subseteq \Pos(S) \subseteq \NE(S)$.
\end{enumerate}
\end{fact}


In the following we will denote $\Neg(S)$ the set of integral curves $C \subset S$ such that $C^2<0$.

\begin{propo}\label{factPos(ix-x)}
If $S$ is a smooth projective surface, we have the following decompositions.
\begin{enumerate}
	\item For any $y \in \NE(S)$, there exist $p \in \Nef(S)$ and $n \in \Eff(S)$ such that $y = p + n$ and $p \cc n =0$.
	\item We have
	\begin{equation}
	\NE(S)= \Pos(S) + \sum_{[C]\in \Neg(S)} R(C) = \Nef(S)+ \sum_{[C]\in \Neg(S)} R(C).
	\end{equation}
\end{enumerate}
\end{propo}
\begin{proof}
To see the first statement, let us consider $y \in \NE(S)$; if $y = [D]$,  where $D$ is a real divisor on $S$, using \cite[Theorem 2.3.19]{laz2},  since the proof of the cited results holds true also for $\RR$-divisors, we get that there is a Zariski decomposition for $D$: $D = P + N$, with $P \in \Nef(S)$ and $N \in \Eff(S)$. The matrix of components of $N$ is definite negative and $P \cdot \Gamma = 0$ for every component $\Gamma$ of $N$.

Setting $p = [P], n = [N]$ we have that $y = [D] = [P] + [N] = p + n$ with $p \in \Nef(S), n \in \Eff(S)$ and $p \cdot n = 0$, that is the first part of Proposition \ref{factPos(ix-x)}.

We now prove the other decomposition. We can see that Fact  \ref{factPos(vi)} immediately gives
\begin{equation}
\NE(S) \supseteq \Pos(S) + \sum_{[C] \in \Neg(S)} R(C) \supseteq \Nef(S) + \sum_{[C] \in \Neg(S)} R(C).
\end{equation}
Viceversa if $y \in \NE(S)$, the first part of the proposition gives $y = p + n$ as above. In particular, since the matrix of the components of $N$ is negative definite, for any component $\Gamma$ of $N$, we have $\Gamma^2 < 0$. It follows that $n = [N] \in \sum_{[C] \in \Neg(S)} R(C)$, and, obviously $y = p + n \in \Nef(S) + \sum_{[C] \in \Neg(S)} R(C)$. This finally gives
$$
\NE(S) \subseteq \Nef(S) + \sum_{[C] \in \Neg(S)} R(C) \subseteq \Pos(S) + \sum_{[C] \in \Neg(S)} R(C).
$$
\end{proof}

\section{The Nagata Conjecture and the $\PP^2$ case}\label{sec:nagata}

In this section we focus on the $\PP^2$ case to stress the relationship between some classical conjectures and some interesting reformulations in terms of Mori theory. This relation has been recently sudied by several authors; we refer in particular to \cite{deF}.


Let us recall that a point of a variety is said to be \emph{general} if it is chosen in the complement of a closed subset and it is said to be \emph{very general} if it is chosen in the complement of the countable union of preassigned proper closed subsets.

Nagata Conjecture (see \cite{Na} or \cite[Remark 5.1.14]{laz2}) is certainly one of the most renowned open problems in the study of planar linear system.

\begin{conjecture}[Nagata Conjecture]\label{cong:nagata}
Let $x_1, \ldots, x_r \in \PP^2$ be very general points; if $r \geqslant 10$, then
\begin{equation}
\deg(D) \geqslant \frac{1}{\sqrt{r}} \sum_{i=1}^r \mult_{x_i}(D)
\end{equation}
for every effective divisor $D$ in $\PP^2$.
\end{conjecture}
A stronger bound is given in the following conjecture.
\begin{conjecture}[see \cite{deF}]\label{conj:nagata2}
Let $x_1, \ldots, x_r \in \PP^2$ very general points; if $r \geqslant 10$, then
\begin{equation}
\deg (D)^2 \geqslant \sum_{i=1}^r \mult_{x_i}(D)^2,
\end{equation}
for every non rational integral curve $D$ in $\PP^2$.
\end{conjecture}

\begin{rem}
In fact, these conjectures can be reformulated in terms of some other conjectures more Mori theory tasting on the blow up of the plane at $r$ points. This reformulation can be made following the spirit of equivalent conjectures of Segre, Harbourne, Gimigliano and Hirschowitz (see \cite{segre}, \cite{harbourne}, \cite{gimigliano} and \cite{hirschowitz}); we will expand this discussion in the following sections and we will get to the statement of the \emph{Segre Problem}.
\end{rem}

Nagata Conjecture has been classically stated for the projective plane; we are interested in some generalization of this kind of statements for $X$, a smooth projective surface $Y$ blown up at $r$ general points. 


We can now ask ourselves some conjecture-like problems: the first of them is about the positive cone $\Pos(X)$ and $K_X$-extremal rays.

\begin{problem}\label{conj:pos}
Let $Y$ be a smooth projective surface and consider $X=\Bl_r(Y)$ the blow up of $Y$ at $r$ very general points, then
\begin{equation}\label{decompo1}
\NE(X) = \Pos(X) + \sum R_i,
\end{equation}
where the sum runs over all $K_X$-negative extremal rays of $\NE(X)$.
\end{problem}

The second, instead, involves curves with self-negative intersection.
\begin{problem}[$(-1)$-Curves Conjecture]\label{conj:-1}
Let $X = \Bl_r (Y)$ the blow up of a smooth projective surface $Y$ and let $C \subset X$ be an integral curve such that $C^2<0$, then $C$ is a $(-1)$-curve.
\end{problem}

\begin{rem}\label{rem:-1}
We just point out that in the case of surfaces, Mori theory gives that if $X$ is a surface with $\rho(X) \geqslant 3$, then extremal rays of $\NE(X)$ spanned by $K_X$-negative curves are precisely those spanned by $(-1)$-curves (see See \cite[Theorem 1.28]{km}).

%
%
%
\end{rem}

Remark \ref{rem:-1} immediately gives that if either $r\geqslant 2$ or $Y\neq \PP^2$ or $Y$ is not minimal ruled, the decomposition in Problem \ref{conj:pos} is equivalent to decomposition
\begin{equation}
\NE(X)= \Pos(X) + \sum_{C \text{ $(-1)$-curve}} R(C).
\end{equation}

It is easy to prove the following.
\begin{fact}\label{equiv1}
Problem \ref{conj:pos} and Problem \ref{conj:-1} are equivalent.
\end{fact}

\begin{rem}
We have seen that Problem \ref{conj:pos} and Problem \ref{conj:-1} are equivalent, but they shall immediately be false if $Y$ contains integral curves $C$ with $C^2 \leqslant -2$. This fact is not so unusual and this is why we didn't use the term \emph{Conjecture} in Problem \ref{conj:pos} and \ref{conj:-1}.
\end{rem}

It is interesting and useful to point out a step toward the proof of Problem \ref{conj:-1} in the case of $Y=\PP^2$: in \cite[Proposition 2.4]{deF2005}, the author shows that if $C$ is an integral rational curve on $X= \Bl_r \PP^2$ with negative self-intersection, then it is a $(-1)$-curve. This proposition allows us to prove the following.
%

\begin{fact}
In the case of $Y=\PP^2$, Problem \ref{conj:pos} is equivalent to Conjecture \ref{conj:nagata2}.
\end{fact}
%
%

\section{The Segre Problem}\label{sec:segre}


In the spirit of the previous section, if $X = \Bl_r Y$ is the blow up of a smooth projective surface $Y$ at $r$ general points, will study some conjectures about the Mori Cone $\NE(X)$ and the curves on $X$ with negative self-intersection. Instead of study linear systems in $Y$ with multiplicities at the $r$ general points $x_1,\ldots,x_r$, we will focus on the linear system $|C|$ associated to an integral curve $C \subset X$.

%
%
In order to generalize the definition of a special linear system $|C|$, we need to require that the dimension $h^2(X,\OO_X(C))=0$. In this situation, indeed, we can give the following definition.
\begin{ddef}
Let $L$ be a line bundle on a smooth projective surface $X$ with $h^2(X,L)=0$; the \emph{virtual dimension} of the linear system $\LL$ associated $L$ is
$v(\LL)= \chi (L) -1$, and its \emph{expected dimension} is $e(\LL)= \max \{ v(\LL), -1\}$.
\end{ddef}

\begin{ddef}\label{def:special}
Let $\LL$ be a linear system on $X$ with $L$ associated line bundle such that $h^2(X,L)=0$.
\begin{itemize}
	\item $L$ is \emph{special} (equivalently $\LL$ is special) if $\dim (\LL)> e(\LL)$;
	\item $L$ is \emph{non special} (equivalently $\LL$ is non special) if $\dim (\LL)= e(\LL)$.
\end{itemize}
\end{ddef}

\begin{ddef}
We say that a linear system $\LL$ on a $X=\Bl_r Y$ is \emph{non exceptional} if there is a divisor in $\LL$ such that its support is not contained in the exceptional locus of $X$.
\end{ddef}


In order to ensure the vanishing of the second cohomology, we restrict to surfaces $Y$ with $p_g(Y)=0$ or $K_Y \equiv 0$. These two cases cover a number of interesting surfaces: in the first we get the projective plane, Enriques surfaces, bielliptic surfaces and a number of surfaces of general type; in the second, namely if $K_Y \equiv 0$ and $p_g(Y)\neq 0$, we have a fortiori that $K_Y \sim 0$ and hence $Y$ has to be an Abelian or a $K3$ surface. The following Lemma can be proved.


%

%
%

\begin{lem}\label{factH2}
Let $Y$ be a smooth surface with either $p_g(Y)=0$, or a $K3$ or an abelian surface.

Let us consider a line bundle $L$ on $X= \Bl_r(Y)$ with associated linear system $\LL \neq \emptyset$. If $\LL$ is not exceptional, then $h^2(X,L)=0$.
\end{lem}

The original Segre Conjecture (see \cite{segre}) about planar linear system can be easily stated for any surface;  in \cite{lafaceK3}, the authors state the Segre Conjecture for a generic $K3$ surface.

\begin{conjecture}[see \cite{lafaceK3}]\label{conjSegK3}
Let $Y$ be a generic $K3$ surface and let $\LL$ be a non empty and reduced linear system on $Y$, then $\LL$ is non special.
\end{conjecture}

More generally, in view of Definition \ref{def:special}, we may consider the following statement.

\begin{problem}\label{segreProblem}
Let $X= \Bl_r Y$ a blown-up surface at $r$ general points; let us suppose $h^2(X,L)=0$ for all line bundles $L$ associated to a non exceptional and non empty linear system $\LL$. If moreover $\LL$ is reduced, then $\LL$ is non special.
\end{problem}

Looking at Lemma \ref{factH2}, we state our formulation of the Segre Problem.

\begin{problem}[Segre Problem]\label{conj:segre}
Let $Y$ be either a $K3$ surface or a surface with $p_g(Y)=0$ or an abelian surface and let $\ffi: X \to Y$ be the blow up at $x_1, \ldots, x_r$, general points of $Y$. If $\LL$ is a non exceptional, non empty and reduced linear system on $X$, then $\LL$ is non special.
\end{problem}

\begin{rem}\label{rem:Abel}
We will soon see in Section \ref{sec:contro} that a statement like the former can't be true for some remarkable cases of surfaces $Y$ with $p_g(Y)=0,$ like Enriques and bielliptic surfaces, and for non simple abelian surfaces, that is an abelian surface not containing any nontrivial abelian subvarieties
%
%
\end{rem}

\subsection{The List Conjecture}
In order to study the consequences of Problem \ref{conj:segre} on $\NE(X)$, let us recall the so called Bounded Negativity Conjecture (see, for example, \cite[Section 1]{global}).


%
%
%


We say that a smooth surface $S$ has \emph{bounded negativity} if there exists an integer $\nu_S$ such that $C^2 \geqslant -\nu_S$ for each integral curve $C \subset S$.
\begin{conjecture}[Bounded Negativity Conjecture]\label{cong:boundedneg}
Every smooth surface $S$ in characteristic $0$ has bounded negativity.
\end{conjecture}

\begin{rem}
It is known that the Bounded Negativity Conjecture is false in positive characteristic: see, for example, \cite[Remark I.2.2]{global}; it may be worth to point out that recent attempts (see \cite{noneg}) to produce counterexamples in characteristic 0 have not been successful: the bounded negativity conjecture remains still open.

\end{rem}

\begin{fact}\label{fact:NegAntiCan}
Bounded Negativity Conjecture holds true for a smooth projective surface $S$ with $-K_S$ pseudoeffective.
\end{fact}
\begin{proof}
Let $C\subset S$ be an integral curve. If $ -K_S \cc C \geqslant 0$, by adjunction, $C^2 \geqslant -2$; if else $-K_S \cc C <0$ then, taking the Zariski decomposition of the pseudoeffective anticanonical divisor, $C$ has to be one of the finitely many components $E_1\ldots, E_s$ of the effective part. Thus $C^2 \geqslant \min\{-2, {E_1}^2, \ldots, {E_s}^2\}$.
\end{proof}


%

Conjecture \ref{cong:boundedneg} suggests the boundedness from below of the self-intersection; on the other hand, in the case of $\PP^2$, the $(-1)$-curves Conjecture not only gives the boundedness of the negativity, but also the boundedness from above of the arithmetic genus.

The $(-n,p)$-curves should thus lie in a sort of list; we state this idea as a conjecture.

%
%
%
\begin{conjecture}[List Conjecture]\label{cong:lista}
Let $C \subset X=\Bl_r Y$ be a non exceptional, integral curve such that $C^2 <0$, then there exist a positive number $\nu=\nu_X$ and a non negative integer $\pi=\pi_X$ such that $C$ is a $(-n,p)$-curve for some $1 \leqslant n \leqslant \nu$ and $0 \leqslant p \leqslant \pi$ (there is a list of possible $(-n,p)$-curves).
\end{conjecture}

Immediately we get the following.
\begin{fact}\label{propList}
Let $X = \Bl_r Y$; if $-K_X$ is pseudoeffective then the List Conjecture \ref{cong:lista} holds true.
\end{fact}
\begin{proof}
Fact \ref{fact:NegAntiCan} gives the existence of a bound $\nu$ on the negativity. Mimicking the proof of Fact \ref{fact:NegAntiCan} we easily get the bound for the arithmetic genus.
\end{proof}

\begin{rem}
In the proof of Fact \ref{propList}, as main ingredient, we used the Zariski decomposition, neverthless we did not use its whole power. To ensure the existence of a finite number of $K$-positive integral curves, we just need the existence of a weak Zariski decomposition of the anticanonical divisor in a nef and an effective part. In view of Fact \ref{propList}, let us take a smooth surface $Y$ with $-K_Y = P+N$; the List conjecture holds true on its blow up $X = \Bl_r Y$ if $\fistar P - \sum_{i=1}^r E_i$ is nef, that is if $\fistar P$ is sufficiently positive and the number of points to blow up is sufficiently small.

It may be worth to point out that this is the case of $Y=\PP^2$ and $r\leqslant 9$.
\end{rem}

\begin{rem}
In view of our main result (see Theorem \ref{mainRes}) it is worth to point out that if we are able to find a smooth projective surface $Y$ with a weak Zariski decomposition for $-K_Y$ with $\fistar P$ sufficiently positive with respect to $r$, then our main result is true, independently from any conjecture, on $X=\Bl_rY$.
\end{rem}


We can now see how a statement like Problem \ref{conj:segre} implies Conjecture \ref{cong:lista} and allows us to find explicit bounds on the negativity and on the arithmetic genus depending only on the blown-up surface $Y$.

\begin{propo}\label{propSegreBound}
Let $C \subset X=\Bl_r Y$ a non exceptional integral curve on a smooth blown-up surface $X$, such that $C^2 <0$; let us suppose Segre Problem (Problem \ref{conj:segre}) has a positive answer.
\begin{enumerate}
	\item It holds: $-1 \geqslant C^2 \geqslant p_a(C)- \chi(\OO_Y) \geqslant -\chi(\OO_Y)$,	and in particular $\chi(\OO_Y)\geqslant 1$;
	\item Conjecture \ref{cong:lista} holds true with $\nu = \chi (\OO_Y), \pi = \chi (\OO_Y)-1$.
\end{enumerate}
\end{propo}
\begin{proof}
Let us consider the linear system $|C|$, associated to $C \subset X$, non exceptional, integral 	curve such that $C^2 <0$; Conjecture \ref{conj:segre} implies that the system is non special, and since it is non empty, we get that $\chi (\OO_X (C)) \geqslant 1$. By Riemann-Roch theorem setting $\chi =\chi (\OO_X)= \chi (\OO_Y) $, we get
$$
C^2 - C \cc K_X \geqslant 2 - 2 \chi.
$$
Hence, by adjunction formula, $C^2-p \geqslant -\chi$. Since $C^2\leqslant -1$ we get: $-1 \geqslant C^2 \geqslant p-\chi \geqslant -\chi$, that is the first point of the proposition. Immediately we find the bounds
\begin{equation}
C^2 \geqslant - \chi (\OO_Y) \qquad \text{and}\qquad p_a(C) \leqslant \chi (\OO_Y) -1.
\end{equation}
\end{proof}

\subsection{Special cases and counterexamples}\label{sec:contro}
In this section we will study the behaviour of the Segre Problem in some special cases; in particular, using elliptic fibrations, we will easily show that it must have a negative answer if the blown-up surface $Y$ is Enriques, bielliptic, or it lies in a subclass of the abelian surfaces.

Consider at first the case of $\chi (\OO_Y)\leqslant 0$. We have the following fact.

\begin{fact}\label{fatto:Ceccez}
Let $Y$ be a smooth projective surfaces with $\chi (\OO_Y)\leqslant 0$ and either $p_g(Y)=0$ or $Y$ is an abelian or $K3$ surface; suppose that Problem \ref{conj:segre} holds true for $X$, the blow up of $Y$ at $r$ general points. If an integral curve $C \subset X$ is such that $C^2<0$, then $C$ is exceptional.
\end{fact}
\begin{proof}
From the first point of Proposition (\ref{propSegreBound}), we see in particular that: $-1 \geqslant C^2 \geqslant - \chi(\OO_Y)$.

If $\chi(\OO_Y)\leqslant 0$, Problem \ref{conj:segre} implies that there can't be non exceptional curves with negative self-intersection.
\end{proof}


\subsubsection*{Projective Plane}

In we consider the $Y = \PP^2$, the projective plane, we have that $\chi (\OO_{\PP^2})= 1$ and Conjecture \ref{cong:lista}, by Proposition \ref{propSegreBound}, gives $\nu =1$ and $\pi = 0$. Therefore we have that if $C$ is an irreducible and reduced non exceptional curve such that $C^2<0$, than $C$ is a $(-1)$-curve; since exceptional curves are $(-1)$-curves, Conjecture \ref{cong:lista} says that on the blow up of the plane at $r$ general points, the only integral curves with negative self-intersection are $(-1)$-curves. Hence we recover the so-called $(-1)$-Curve Conjecture (see Problem \ref{conj:-1} or \cite[Conjecture 1.1]{deF}).

%

\subsubsection*{Surfaces with a fibration and easy counterexamples}

We want now to put in evidence some easy counterexamples to the Segre Problem; in particular we will focus on surfaces $Y$ having a sort of fibration.

\begin{fact}\label{factNoSegre}
Let $Y$ be a surface with a base point free pencil $V$ of curves of arithmetic genus $g$; if for the strict transform $\tilde C\subset X = \Bl_r Y$ of a general curve in the pencil we have
\begin{equation}
\chi(\OO_Y)  \neq \dim |\tilde C| +g +1,
\end{equation}
then the Segre Problem has a negative answer for $X= \Bl_r Y$.

In particular, in the $p_g(Y)=0$ case,  there is a negative answer if  $g >0$ or $q>0$, where $g$ is the genus of the curves in the pencil and $q$ is the irregularity of $Y$.
\end{fact}
\begin{proof}
The base point free pencil $V$ determines a morphism $\psi_V : Y \to \PP^1$, whose fibres are exactly the genus $g$ curves of $V$. Thus for $C \in V$, we immediately get $C^2 =0$.

Let us focus on $X=\Bl_r Y$ and let us consider the fibre $C=F_1$ passing through $x_1$. By generality and by Bertini theorem, we can suppose that $C$ is a smooth curve and therefore
$$
m_1 = \mult_{x_1}(C) = 1; \qquad
m_i = \mult_{x_i}(C)=0, \quad \text{ for }i=2, \ldots, r.
$$
Moreover, we can also see that the curve $C$ is irreducible and hence integral; thus its strict transform $\tilde C = \fistar C - E_1$ is an integral, smooth and non exceptional curve with $\tilde C^2 =-1$.

Let us suppose that the Segre problem holds true; for $L =\OO_X(\tilde C)$, Segre problem gives $\dim(|L|)=\max\{\chi(L)-1,-1\}$; since $|L|\neq \emptyset$, then
\begin{equation}\label{eq:Pip}
 \dim(|L|)=\chi(L)-1.
 \end{equation}
 By Riemann-Roch theorem and adjunction, we get
$$
\chi(L) = \chi(\OO_X) + \frac{1}{2}(\tilde{C}^2- \tilde C \cc K_X) =\chi(\OO_Y)+\tilde{C}^2 -g +1.
$$
Now, by equation \eqref{eq:Pip}, we get $\chi(\OO_Y) = \dim (|L|) + g +1$, a contradiction with our ad hoc hypothesis. Since $\dim(|L|) \geqslant 0$, we get the bound $\chi(\OO_Y) \geqslant g+1$.

In the $p_g (Y)=0$ case, since $\chi(\OO_Y)= 1-q$, this becomes $g+q \leqslant 0$ and hence $q=g=0$. Thus, whenever $g >0$ or $q >0$, the Segre problem has a negative answer.
\end{proof}

Let us recall that a surface $Y$ has an \emph{elliptic fibration} (see \cite{bpvdv}) if there exists a proper connected morphism $Y \to C$ to an algebraic curve $C$ such that the general fibre is a smooth elliptic curve. It is immediate to state the following fact.
\begin{fact}\label{noSegreEnBi}
Let $Y$ be either an Enriques or a bielliptic surface or a non simple abelian surface, then the Segre Problem for $X= \Bl_r Y$ has a negative answer.
\end{fact}
\begin{proof}
It is enough to find an elliptic curve $C \subset Y$ passing through the first blowing-up point $x_1$. In the case of an Enriques or a bielliptic surface this is a consequence of the existence of an elliptic fibration.

In the non simple abelian case, we have that there is an elliptic curve $D \subset Y$ and by \cite[Poincaré's complete reducibility theorem]{birlan} there exists another elliptic curve $D'$ such that $Y$ is isogeneous to $D \times D'$; thus in particular there is an elliptic curve passing through a point and, by translation, there is such a curve through any point.

Let us take the elliptic curve $C$ passing through the first blowing-up point $x_1$. We immediately see that $C^2 =0$; by generality of the points, we can suppose that  $m_1=1$ and $m_i=0$ for $i=2,\ldots, r$. For the strict transform $\tilde C$, we get $\tilde{C}= \fistar C-E_1$, which gives $\tilde C^2 =-1$. Now, since $\tilde C$ is a non exceptional curve with negative self-intersection, if Segre Problem had a positive answer, Proposition \ref{propSegreBound} would give
$$
-1 \geqslant \tilde{C}^2 \geqslant 1 - \chi(\OO_Y),
$$
hence $\chi(\OO_Y) \geqslant 2$, that is a contradiction.
\end{proof}

\section{Negative part of the Mori cone}\label{sec:negative}
If $X$ is a smooth projective surface, we want to study the decomposition of $\Neg(X)$, the set of integral curves $C \subset X$ such that $C^2<0$. In particular, this decomposition allows us to study the structure of the Mori cone $\NE(X)$. In view of Conjecture \ref{cong:lista}, we have the following proposition.
\begin{propo}\label{propo:list}
Let $X$ be a smooth projective surface; let us consider a finite subset $
L'\subset \left( (-\infty, -1] \cap \ZZ \right) \times \left( [0,+\infty) \cap \ZZ \right)$; we say that the integral curve $C \subset X$ is in the list $L'$ if $(C^2,p_a(C))\in L'$. Let $L =\left\{ [C] \mid (C^2,p_a(C)) \in L'\right\}\subset N(X)$, then the following are equivalent:
\begin{enumerate}
	\item for all integral curve $C \subset X$ such that $C^2<0$ we have that $[C] \in L$;
	\item we have the decomposition $\NE(X)= \Pos(X)+ \sum_{[C]\in L} R(C)$.
\end{enumerate}
\end{propo}
\begin{proof}
Let us suppose the first, then the second is an easy consequence of Proposition \ref{factPos(ix-x)}.

To prove the reverse implication. Consider an integral curve $C$ such that $C^2 <0$; by hypothesis we get the decomposition: $[C]= \alpha + \sum_{i \in I} b_i [C_i]$, where $\alpha \in \Pos(X)$, $b_i > 0$ and $[C_i]\in L$. Now, since $C^2 <0$, $[C]$ spans the extremal ray $R(C)$. By extremality we have that $\alpha \in R(C)$ and so there exists a real number $a \geqslant 0$ such that $\alpha = a[C]$. We immediately get: $0 \leqslant \alpha^2 = a^2 [C]^2 \leqslant 0$, which gives $a^2 [C]^2 =0$ and so $a=0$ and $\alpha =0$. Again by extremality, we also have that $[C] \in R(C_i)$ for all $i\in I$; but since in such a ray there can't be two distinct integral curves, the decomposition has only a summand with $b_{i_0}=1$, $C=C_{i_0}$ and in particular $[C]\in L$.
\end{proof}

\begin{rem}
Since we are considering blown-up surfaces at $r$ points, we can't avoid the exceptional curves $E_1, \ldots, E_r \subset X$. In light of this, it is immediate to see that the first claim in Proposition \ref{propo:list} is thus equivalent to Conjecture \ref{cong:lista}. In particular a positive answer to Segre Problem implies the decomposition given in the second statement.
\end{rem}

\subsection{K3 surfaces}

The case of $K3$ surfaces has been considered in \cite{lafaceK3}; in this paper the authors state the Segre Conjecture for a generic $K3$ surface, that is $Y$ is a $K3$ and $\Pic(Y)= \ZZ [h]$, for an ample class $h$ on $Y$. Let $Y$ be a $K3$ surface and let $X =\Bl_rY$; we want to study in more details how Segre Problem forces the structure of the negative part of $\NE(X)$. We have the following.

\begin{fact}
Let $X= \Bl_r Y$ the blow up of a $K3$ surface; if the Segre Problem has a positive answer, then the list in Proposition \ref{propo:list} is given by curves of kind I, II or III and we have the decomposition:
$$
\NE(X)= \Pos(X) + \sum_{C_i \text{ of kind I}} R(C_i) + \sum_{C_j \text{ of kind II}} R(C_j) + \sum_{C_k \text{ of kind III}} R(C_k),
$$
where the curves with negative self-intersetion are of one of the following kind.
\end{fact}

\begin{center}
\begin{tabular}{c | ccc | c | c }
 & &kind I & & kind II & kind III \\
\hline
&&&&&	\\
$(C^2, p_a(C))$ && (-1,0)& & (-2,0) & (-1,1) \\
$C \cc K_X$ && -1 && 0 & 1 \\
$\ffi(C)$ && point && $\Gamma$ & $\Gamma$ \\
$(\Gamma^2, p_a(\Gamma))$ &&& &(-2,0) & (0,1) \\
&&& & $\mult_{P_1}(\Gamma)=0$ & $\mult_{P_i}=1, mult_{P_j}=0$ \\
& &&& for all $i$ & for all $j \neq i$.
\end{tabular}
\end{center}

In the case of generic $K3$, we have the following fact.
\begin{fact}
If $Y$ is a generic $K3$ surface; suppose Segre Problem has a positive answer for $X = \Bl_r Y$, then if $C$ is an irreducible curve such that $C^2 <0$, then it is an exceptional $(-1)$-curve.
\end{fact}
\begin{proof}
Since $Y$ is generic, then $\Pic(Y)= \ZZ[h]$ and $\NE(Y)= R(h)$ is simply the ray generated by $h$. Therefore for every curve on $Y$ we have $C^2 >0$, hence on $X$ there can't be curves of kind II or III.
\end{proof}

\section{Circular part of the Mori cone}\label{sec:circular}

In view of the conjectures from the former sections, it is reasonable to ask ourselves the following.
\begin{problem}\label{conj:posCircPart}
Let $X$ be the blow up of a smooth algebraic surface at $r$ (eventually large) general points. Then there exists an $\RR$-divisor $D$ on $X$ such that $\NE(X)_{D^{\geqslant 0}} \neq \{0\}$ and
\begin{equation}
\NE(X)_{D^{\geqslant 0}} = \Pos(X)_{D^{\geqslant 0}}.
\end{equation}
\end{problem}

In Theorem \ref{mainRes} we will derive the solution to this problem as a consequence of the List Conjecture (see Conjecture \ref{cong:lista}), supposing $r$ sufficiently large and supposing the bounds depending only on $Y$; in particular, since this is assured by Segre Problem, Problem \ref{conj:posCircPart} would follow from Segre.



As before, we work with a smooth projective surface $Y$ with either $p_g(Y)=0$ or $Y$ a $K3$ or an abelian surface and let $\varphi: X = \Bl_r Y \to Y$ be the blow up at the general points  $x_1, \ldots, x_r$. 


From now on, $A$ will be a fixed ample divisor on $Y$ and $L= \varphi^* A$ its nef pullback to $X$; we will denote $K=K_X$ the canonical divisor of $X$.

\begin{fact}\label{conticini} We have: ${K_X}^2= {K_Y}^2 -r$, $L^2= A^2 >0$ and $K_X \cdot L = A \cdot K_Y$.
\end{fact}



We want to find conditions on the (eventually large) number $r$ of points to blow up in order to describe the Mori Cone $\NE(X)$ of the blown-up surface $X$ in terms of the positive cone $\Pos(X)$ (see Proposition \ref{factPos(ix-x)}).
%

Here is our strategy: we fix an integral curve $C$ generating a $(-n,p)$-ray and then we find an $s =s(n,p)\in \RR$ such that $R(C) \subset \Pos(X) + R(K - s L)$. Performing our program we will find some inequalities on the number of points $r$; the bounded negativity condition will allow us to avoid accumulation phenomena. Let us first prove the fact in the case of $(-1,p)$-curves.



\begin{propo}[$(-1,p)$-case]\label{(-1)prop}
Let $Y$ be an smooth projective surface and $X = \Bl_r Y$ the blow up of $Y$ at $r$ general points. If $R$ is a $(-1,p)$-ray generated by a curve $C$ and we have
\begin{equation}\label{condEnunciato(1p)}
\begin{cases}
r > {K_Y}^2 + 1 - \frac{(A \cc K_Y)^2}{A^2}\\
r \geqslant {K_Y}^2 +1 + 4 A^2 p^2 - 4(A \cc K_Y)p, & \quad \text{if } p > \frac{A \cc K_Y}{2A^2},
\end{cases}
\end{equation}
then there exists $s_1=\frac{A\cc K_Y + \sqrt{(A \cc K_Y)^2 - A^2 {K_Y}^2+A^2 r - A^2}}{A^2}$, such that $R(C) \subset \Pos(X) + R(K-s_1 L)$.
\end{propo}

\begin{proof}
As first step we want to find a positive solution for $t$ of the equation
\begin{equation}\label{eq1}
(tC -(K-sL))^2=0,
\end{equation}

where $C$ is the $(-1,p)$-curve generating $R$. To ensure the existence of solutions of (\ref{eq1}) we need $\Delta \geqslant 0$. Since, by adjunction, $C \cc K=2p-1$, we have to ask
$$
\Delta/4 = (2p-1-sC\cdot L)^2 + (K-sL)^2 \geqslant 0.
$$
%
%
%
%

To this end it is enough to require the existence of $s$ such that
\begin{equation}\label{cond[D](-1,p)}
(K-sL)^2=-1 \qquad \text{and}\qquad (2p-1-s C \cc L)^2 \geqslant 1.
\end{equation}
The first equation, by Fact \ref{conticini}, becomes $A^2 s^2 -2s A \cdot K_Y + K_Y^2-r+1=0$, and it has solutions if its discriminant $\Delta_1/4 = (A\cc K_Y)^2-A^2 K_Y^2 + A^2 r -A^2 \geqslant 0$, that is if $r \geqslant {K_Y}^2 + 1 -\frac{(A \cdot K_Y)^2}{A^2}$.

Since in the following we will need the strict positivity of this discriminant, our first numerical condition on the number of points to blow up is:

\begin{equation}\label{cond[1](-1,p)}
r > {K_Y}^2 + 1 -\frac{(A \cdot K_Y)^2}{A^2}.
\end{equation}

In this situation we can take 
\begin{equation}\label{esse}
s_1 = \frac{A \cdot K_Y + \sqrt{\Delta_1 / 4}}{A^2};
\end{equation}
let us note as $s_1$ do not depend on the specific curve $C$, but just, as will be clearer in the next proposition, on the value of $C^2$. Now let us fix $s=s_1$ as in (\ref{esse}) and let us check the second inequality in (\ref{cond[D](-1,p)}). We immetiately see that it is enough that
\begin{equation}\label{conto[1](-1,p)}
2p-1-s C \cc L \leqslant -1.
\end{equation}
Now, if $C\cdot L =0$ than $C$ is a contracted curve and so is one of the exceptional divisors $E_i$; in particular we have $p=0$ and so the inequality holds. If else $C \cc L >0$, the condition \eqref{conto[1](-1,p)} gives $s \geqslant \frac{2p}{C \cc L}$. Since $C \cc L \geqslant 1$, it is enough to have $s \geqslant 2p$; this is true when
$$
s=\frac{A \cc K_Y + \sqrt{\Delta_1/4}}{A^2} \geqslant 2p,
$$
which gives $\sqrt{\Delta_1/4} \geqslant  2 A^2 p- A \cc K_Y$. If the right hand side is non positive, that is when $p \leqslant A \cc K_Y / 2 A^2$, the inequality holds true and we have no other conditions to impose. If otherwise $p > A \cc K_Y / 2 A^2$, we get
$$
(A \cc K_Y)^2 - A^2 K_Y^2 + A^2 r - A^2 \geqslant 4 (A^2)^2 p^2 + (A \cc K_Y)^2- 4 A^2(A \cc K_Y)p,
$$
that gives $r \geqslant {K_Y}^2 +1 + 4 A^2 p^2 - 4(A \cc K_Y)p$. Hence, we have the two conditions:
$$
\begin{cases}
r > {K_Y}^2 + 1 - \frac{(A \cc K_Y)^2}{A^2}\\
r \geqslant {K_Y}^2 +1 + 4 A^2 p^2 - 4(A \cc K_Y)p & \quad \text{if } p > \frac{A \cc K_Y}{2A^2}.
\end{cases}
$$

In this situation the \eqref{cond[D](-1,p)} holds true and we have solutions of $(tC - (K-sL))^2=0$. We can fix one of the two solution:
\begin{equation}
t_0= -C \cc(K-sL) + \sqrt{\Delta/4}= - (2p-1-sC\cc L) + \sqrt{\Delta /4}.
\end{equation}
Thanks to (\ref{conto[1](-1,p)}), we have that $t_0 \geqslant 1 >0$ and so we get a positive solution of (\ref{eq1}).

Now we have that $\alpha = t_0 C - (K-sL)$ satisfies $\alpha^2 = 0$. In order to prove that $\alpha \in \Pos(X)$, we need to check that $\alpha \cdot h \geqslant 0$ for some $h$ ample.

Since if $\alpha \cc h \geqslant 0$ then $\alpha \cc h' \geqslant 0$ for any other ample class $h'$, we can set $h= L - \sum \delta_i E_i$, with $\delta_i >0$; we will fix the $\delta_i \ll 1$ after the following formal computation. A little remark: since we have the strict positivity in \ref{cond[1](-1,p)}, then $\sqrt{\Delta_1/4}>0$. It is immediate to see that

%
%
%




\begin{equation}\label{contoampio}
\alpha \cc h = [t_0 C - (K - sL)] \cdot [L - \sum \delta_i E_i] 
 = t_0 C \cdot L-  t_0 \sum \delta_i E_i \cdot C +\sqrt{\Delta_1/4} - \sum \delta_i.
\end{equation}

Now, since $t_0 C \cc L + \sqrt{\Delta_1/4}>0$ and $E_i \cc C$ depends only on $C$, for any $C$, we can fix small $\delta_i$ for which $\alpha \cc h \geqslant 0$.

We have hence that for a positive $t_0$, $\alpha = t_0 C - (K-sL) \in \Pos(X)$. Therefore $t_0 C \in \Pos(X) + (K-sL)$ and so, since $t_0$ is positive, $R(C) \subset \Pos(X) + R(K -sL)$.
\end{proof}

Our goal is now to prove of a similar fact in the general case of $(-n,p)$-curves.

\begin{propo}[$(-n,p)$-case]\label{(-n,p)prop}
Let $Y$ be an algebraic projective smooth surface and $X = \Bl_r Y$ the blow up of $Y$ at $r$ general points. If $R$ is an $(-n,p)$-ray, $n \geqslant 2$, generated by a curve $C$, let $q=2p+n-1$,  and let us suppose that
\begin{equation}\label{condEnunciato(np)}
\begin{cases}
r \geqslant {K_Y}^2 + \frac{1}{n} - \frac{(A \cc K_Y)^2}{A^2}\\
r \geqslant {K_Y}^2 + \frac{1}{n} + A^2q^2-2(A \cc K_Y)q & \text{if }q > 	\frac{A \cc K_Y}{A^2},
\end{cases}
\end{equation}
then there exists $s_n=\frac{A\cc K_Y + \sqrt{(A \cc K_Y)^2 - A^2 {K_Y}^2+A^2 r - A^2/n}}{A^2}$, such that $R(C) \subset \Pos(X) + R(K-s_{n} L)$.
\end{propo}

\begin{proof}
As before, we want to find a positive solution of the equation
\begin{equation}\label{eqN}
(tC -(K-sL))^2=0,
\end{equation}

where $C$ is the $(-n,p)$-curve generating $R$. To ensure the existence of solutions of (\ref{eqN}) we need $\Delta \geqslant 0$ that is
%
$$
(2p+n-2-sC\cdot L)^2 + n(K-sL)^2 \geqslant 0.
$$
To have this, it is enough require the existence of a $s$ such that
\begin{equation}\label{cond[D](-n,p)}
(K-sL)^2=-\frac{1}{n} \qquad \text{and} \qquad (2p+n-2-s C \cc L)^2 \geqslant 1.
\end{equation}
As in the former proposition, the first of \eqref{cond[D](-n,p)} has solution if its discriminant 
%
%
%
%
$$
\frac{\Delta_n}{4}:= (A \cdot K_Y)^2 -A^2 K_Y^2 + A^2 r -\frac{A^2}{n} \geqslant 0,
$$
that is if
\begin{equation}\label{cond[1](-n,p)}
r \geqslant {K_Y}^2 + \frac{1}{n} -\frac{(A \cdot K_Y)^2}{A^2}.
\end{equation}
In this situation we can take 
\begin{equation}\label{esseN}
s_n = \frac{A \cdot K_Y + \sqrt{\Delta_n / 4}}{A^2},
\end{equation}

For the second inequality in (\ref{cond[D](-n,p)}), with $s=s_n$, it is enough to check that

%
%
\begin{equation}\label{conto[1](-n,p)}
2p+n-2-s C \cc L \leqslant -1.
\end{equation}
To this end, since $C \cc L \geqslant 1$, it is enough to ask that $s \geqslant 2p+n-1$.

%

Using the definition of $s=s_n$, we immediately get $\sqrt{\Delta_n/4} \geqslant A^2(2p+n-1)- A \cc K_Y$

Let us set $q:=2p+n-1$; if $q<\frac{A \cc K_Y}{A^2}$ we have no other condition to impose; otherwise we get

$$
\Delta_n/4 \geqslant (A^2q- A \cc K_Y)^2
$$
which gives
\begin{equation}\label{cond[3](-n,p)}
r \geqslant {K_Y}^2 +\frac{1}{n} + A^2 q^2 - 2(A \cc K_Y)q.
\end{equation}

In this situation one of the two solutions of $(tC - (K-sL))^2=0$ is
\begin{equation}
t_0=\frac{-(2p+n-2-s C\cc L) + \sqrt{(2p+n-2-s C\cc L)^2-1}}{n}.
\end{equation}
Thanks to the choice we did in (\ref{conto[1](-n,p)}) we have that $t_0 \geqslant 1/n >0$ and so we have a positive solution of (\ref{eqN}). Now we have that $\alpha = t_0 C - (K-sL)$ such that $\alpha^2 =0$. Let us check that $\alpha \cdot h \geqslant 0$ for some $h$ ample. Mimicking \eqref{contoampio}, we get:
%


\begin{equation}\label{contoampioN}
\alpha \cc h =[t_0 C - (K - sL)] \cdot [L - \sum \delta_i E_i]= \underbrace{t_0 C \cdot L}_{> 0}-  t_0 \sum \delta_i E_i \cdot C +\underbrace{\sqrt{\Delta_n/4}}_{\geqslant 0} - \sum \delta_i.
\end{equation}
Now, since $t_o C \cc L + \sqrt{\Delta_n/4} >0$, we can fix some small $\delta_i$s (eventually depending on $C$) such that $\alpha \cc h$ is positive. Again we have that for positive $t_0$, $\alpha = t_0 C - (K-sL) \in \Pos(X)$. Therefore $t_0 C \in \Pos(X) + (K-sL)$ and hence $R(C) \subset \Pos(X) + R(K -sL)$.
\end{proof}

%


%
Now, in view of what we pointed out at the beginning of this section, if we suppose the List Conjecture, we have $-1 \geqslant C^2 \geqslant -\nu$ and  $0 \leqslant p_a(C) \leqslant \pi$, for every integral curve $C$ with negative self-intersection and integers $\nu$ and $\pi$ depending only on $Y$.

%
In this situation we need to solve, for $n=1, \ldots, \nu$ and $p=0, \ldots, \pi$, the inequalities \eqref{condEnunciato(1p)} and \eqref{condEnunciato(np)}; these are verified if $r > {K_Y}^2 + 1 -\frac{(A\cc K_Y)^2}{A^2}$, and, in the case $2\pi + \nu -1 > \frac{A \cc K_Y}{A^2}$, if $r \geqslant {K_Y}^2 + 1 + A^2 (2 \pi + \nu - 1)^2 - 2(A \cc K_Y)(2 \pi + \nu -1)$.
It is easy to see that the second implies the first, hence, setting $q = 2\pi +\nu-1$, our conditions can be summarized in
\begin{equation}\label{condTOT}
\begin{cases}
r > {K_Y}^2 + 1 -\frac{(A\cc K_Y)^2}{A^2} & \text{if} \quad q \leqslant \frac{A\cc K_Y}{A^2}\\
r \geqslant {K_Y}^2 + 1 + A^2 q^2 - 2(A \cc K_Y)q & \text{if}\quad q > \frac{A\cc K_Y}{A^2}.
\end{cases}
\end{equation}


%
It is obvious to point out, looking at the definitions, that $s_1$ is the smallest and $s_{\nu}$ is the largest: $s_1 < s_2 < \cdots < s_{\nu}$. This seems interesting because of the following fact:

\begin{fact}\label{fattos1s2}
In our situation, if $s \geqslant t$ we have that $
(K-s L)^{\perp} \cap L^{\geqslant 0} \subset (K-t L)^{\geqslant 0} \cap L^{\geqslant 0}$. In particular, since $L$ is nef, this is true intersecting with $\NE(X)$ instead of $L^{\geqslant 0}$.
\end{fact}
\begin{proof}
Let $\gamma \in (K-s L)^{\perp} \cap L^{\geqslant 0}$, we get $\gamma \cc K - s \gamma \cc L =0$ that gives $\gamma \cc K = s \gamma \cc L$. Hence: $
\gamma \cc (K - t L)= \gamma \cdot K - t \gamma \cc L = s \gamma \cc L - t \gamma \cc L= (s-t)\gamma \cdot L \geqslant 0$.
\end{proof}

\begin{fact}\label{fact:rneg}
If the conditions \eqref{condTOT} are verified, there is an ample class $h= L - \sum_i \delta_i E_i$, with $0<\delta_i \ll 1$, such that $(K - s L) \cc h <0$ for all $s=s_n=\frac{A \cc K_Y+\sqrt{\Delta_n/4}}{A^2}$.
\end{fact}
\begin{proof}
Let us compute.
\begin{equation}\label{eq:raggineg}
\begin{aligned}
&(K-sL)\cdot (L-\sum_i \delta_i E_i)= (K-sL)\cdot L - (K-sL)\cdot (\sum_i \delta_i E_i)=\\
&K_Y \cc A - sA^2 -\sum_i \delta_i K \cc E_i+ s \sum_i \delta_i L\cc E_i  =\\
&K_Y \cc A - sA^2 + \sum_i \delta_i  = K_Y \cc A -\frac{K_Y \cc A + \sqrt{\Delta_n/4}}{A^2}A^2 + \sum_i \delta_i  = -\sqrt{\Delta_n/4} + \sum_i \delta_i;
\end{aligned}
\end{equation}
that is negative since $\Delta_n/4 >0$ and $\sum_i \delta_i$ is small.
\end{proof}


We are now getting closer to our main result; we need some other preliminary  results.

\begin{fact}
For all $t\neq s \in \RR$ we have $\left( (K-sL)^{\perp} \cap \PosO(X) \right) \cap \left( (K-tL)^{\perp} \cap \PosO(X) \right) = \emptyset$.

\end{fact}
\begin{proof}
Consider $\gamma$ in the intersection, then $(K-sL) \cc \gamma =0 = (K-tL) \cc \gamma$, that is $(t-s)L \cc \gamma =0$,  but since $t \neq s$ this means $L \cc \gamma =0$, but this is impossible since $L$ is nef and $L^\perp$ lies outside $\PosO(X)$.
\end{proof}
We are now able to give the following proposition


\begin{propo}\label{tuttiContenuti}
If $s \geqslant t$, then $\Pos(X) + R(K -t L) \subset \Pos(X) + R(K - s L)$. In particular, if $C$ is a $(-n,p)$-curve, for some $0< n \leqslant \nu$ and $0 \leqslant p \leqslant \pi$, then
$$
R(C) \subset \Pos(X) + R(K-s_{\nu} L).
$$
\end{propo}
\begin{proof}
Let us consider $\gamma \in \Pos(X) + R(K-tL)$; we can write $\gamma = \alpha + a(K-tL)$ and $\alpha \in \Pos(X), a \geqslant 0$. We have
$$
\gamma =\alpha + a(K -sL + sL - tL) =\alpha + a(s-t)L + a (K-sL) \in \Pos(X) + R(K-sL),
$$
since $s \geqslant t$, $L$ is nef and hence it lies in $\Pos(X)$. Recalling the results of Proposition \ref{(-1)prop} and Proposition \ref{(-n,p)prop}, since $s_1 < s_2 < \cdots < s_{\nu}$, we immediately get the second statement.
\end{proof}
 
In the case of $\rho(X)=3$, the situation in Proposition \ref{tuttiContenuti},  can be pictured as in Figure \ref{-ncurve}. In particular we can see that as $s=s_n$ grows, the ray $R(-(K-sL))$ get closer to the boundary of $\Pos(X)$.

\begin{figure}
  \centering
  \includegraphics[scale=0.75]{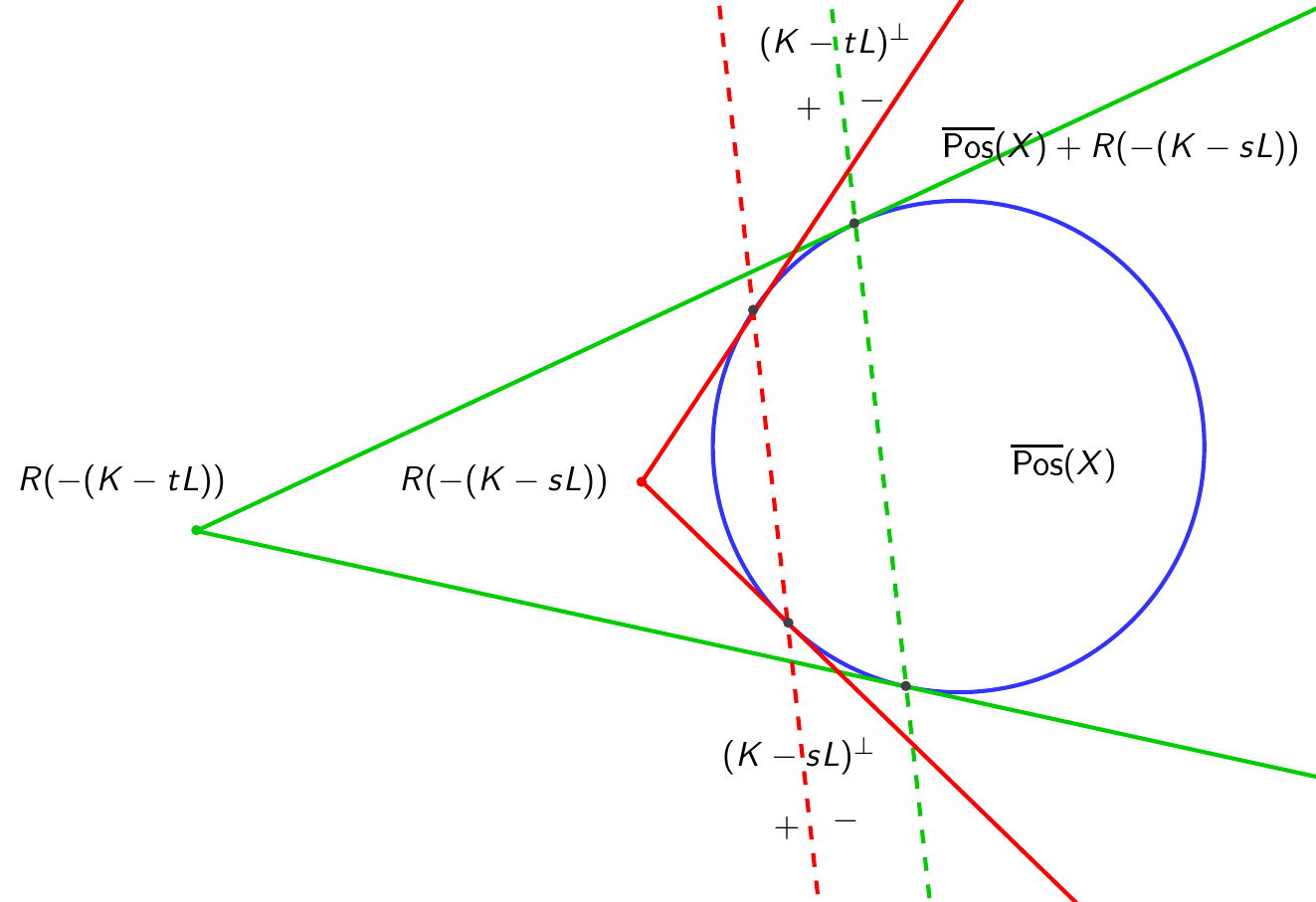}
  \caption{The positive cone $\Pos(X)$ and the behaviour of $R(-(K-sL))$}
  \label{-ncurve}
\end{figure}



We are now ready to state our main result. We prove that Problem \ref{conj:posCircPart} has a positive answer if the List Conjecture is true with bounds depending only on $Y$ and the number of points $r$ is sufficiently large. In particular this is a consequence of Segre Conjecture.


\begin{teo}\label{mainRes}
Let $\varphi: X \to Y$ the blow up at a set of $r$ general points of a smooth projective surface $Y$. Let $A$ be an ample divisor on $Y$ and $L = \varphi^* A$. Let us suppose that:
\begin{enumerate}
	\item there exist two integer numbers $\nu= \nu_X$ and $\pi= \pi_X$ such that the List Conjecture holds on $X$ with bounds for $(-n,p)$-curves given by $1 \leqslant n \leqslant \nu$ and $0 \leqslant p \leqslant \pi$; this is verified, for example, if Segre Problem holds true on $X$ (see Proposition \ref{propSegreBound}) or if $-K_X$ is pseudoeffective (see Proposition \ref{propList}).
%
	\item the following inequalities, with $q= 2\pi +\nu-1$, hold:
\begin{equation}\label{condizTeo}
\begin{cases}
r > {K_Y}^2 + 1 -\frac{(A\cc K_Y)^2}{A^2} & \text{if} \quad q \leqslant \frac{A\cc K_Y}{A^2}\\
r \geqslant {K_Y}^2 + 1 + A^2 q^2 - 2(A \cc K_Y)q & \text{if}\quad q > \frac{A\cc K_Y}{A^2}.
\end{cases}
\end{equation}
%
%
%
\end{enumerate}

Then there exists $s=s_{\nu} \in \RR$,
\begin{equation}
s=\frac{(A \cc K_Y)+ \sqrt{(A \cc K_Y)^2-A^2 {K_Y}^2+A^2 r -A^2/\nu}}{A^2},
\end{equation}
such that
\begin{equation}
\NE(X)_{(K-sL)^{\geqslant 0}} = \Pos(X)_{(K-sL)^{\geqslant 0}}.
\end{equation}
That is, Problem \ref{conj:posCircPart} is true with $D= K-sL$.

In particular, conditions $1.$ and $2.$ are verified, for $r \gg 0$, if the bounds $\nu$, $\pi$ depend only on $Y$.
\end{teo}

\begin{proof}
Since the $\rho(X)\leqslant 2$ case is trivial, we focus on $\rho(X)\geqslant 3$. Let $R=R(C)$ be an extremal ray of $\NE(X)$ spanned by the class of the irreducible curve $C$ with $C^2<0$. Since we are assuming the List conjecture, we have that $C^2 \geqslant -\nu$, for some integer $\nu$. Then by Proposition \ref{tuttiContenuti}, we have that $R = R(C) \subset \Pos(X) + R(K-sL)$, where $s=s_{\nu}$ is the real number constructed in Proposition \ref{(-1)prop} or in Proposition \ref{(-n,p)prop}.



\begin{claim}\label{claimG}
Let $G = R^{\perp} \cap \Pos(X)$, then $G \subseteq (K-sL)^{\leqslant 0}$.
\end{claim}
\begin{proof}[Proof of the Claim]
Let us take $\gamma \in G$ and $0 \neq \delta \in R$; in particular $\gamma \cc R=0$ and since $R \subset \Pos(X)+ R(K-sL)$, we can write $\delta = \alpha + a(K-sL)$, with $\alpha \in \Pos(X)$ and $a >0$. 

We can compute: $0 = \gamma \cc \delta = \gamma \cc \alpha + a \gamma \cc (K-sL)$, which gives $a \gamma \cc (K-sL) = - \gamma \cc \alpha$, that is non positive, since $\gamma,\alpha \in \Pos(X)$ and, by Fact \ref{factPos(iii)}, $\gamma \cc \alpha \geqslant 0$.
\end{proof}

%

Now a well-known theorem by Campana and Peternell gives a description of the shape of  $\partial \Nef(X)$, see, for example \cite[Theorem 1.5.28]{laz2}:
\begin{equation}\label{camPet}
\partial \Nef(X) \subseteq \partial \Pos(X) \cup \left( \bigcup_i H_i \right),
\end{equation}
with $H_i =C_i^{\perp}$ for some integral $C_i$ with $C_i^2 <0$.

%

\begin{claim}\label{claimNefPos}
$\partial \Nef(X)_{(K-sL)^{>0}} = \partial \Pos(X)_{(K-sL)^{>0}} $.
\end{claim}
\begin{proof}[Proof of the Claim]
To prove the first inclusion we see that by Claim \ref{claimG}, we have that if $\beta \in \partial \Nef(X)$ is supported on an hyperplane, then $\beta \in (K-sL)^{\leqslant 0}$ and hence, by \eqref{camPet}, $\partial \Nef(X) \cap (K-sL)^{>0} \subseteq \partial \Pos(X) \cap (K-sL)^{>0}$.

Let us now focus on the reverse inclusion and let us consider $0\neq \alpha \in \partial \Pos(X) \cap (K-sL)^{>0}$; it is enough to show that $\alpha \in \Nef(X)$.
	
%
%

	

	Suppose by contradiction that $\alpha$ is not nef; then there exists a class of a curve $C$ such that $\alpha \cc C <0$ and it must be $C^2 < 0$ (else we would have $\alpha \cc C \geqslant 0$ by Fact \ref{factPos(iii)}). 
	
	Setting $G= C^{\perp} \cap \Pos(X)$, as in Claim \ref{claimG}, we get $G \subseteq (K-sL)^{\leqslant 0}$; since $(K-sL) \cc C <0$ (see equation \eqref{conto[1](-n,p)}), we immediately get
	
	


	\begin{equation}\label{eq:GCKsL}
	G + R(C) \subset (K-sL)^{\leqslant 0};
	\end{equation}

%
%
	We now claim the following:
	\begin{equation}\label{claim:span}
	C^{\leqslant 0} \cap \Pos(X) \subseteq G + R(C).
	\end{equation}
	To prove it, let us take $0 \neq \beta \in C^{\leqslant 0} \cap \Pos(X)$ we can suppose that $\beta \cc h = C \cc h$. Now if\linebreak $\beta \cc C=0$ we are done. If $\beta \cc C<0$, we claim that $\beta \cc C -C^2 >0$.
	
	 Indeed, since $\beta \in \Pos(X)$, we have $0 \leqslant \beta^2=(\beta -C)^2 + (\beta -C)\cc C+ \beta \cc C$, which gives $\beta \cc C -C^2 = (\beta -C)\cc C \geqslant -\beta \cc C - (\beta -C)^2$. We claim that this is positive since $\beta \cc C <0$ and $(\beta -C)^2 <0$ (it easy to see that if it were $(\beta -C)^2 \geqslant 0$, then $\beta = C$).

	

%
%
	


	Now, the line $L=\{ t \beta + (1-t)	C \mid t \in \RR\}$ joining $C$ and $\beta$ does intersect $C^{\perp}$ in the point $\gamma$ corresponding to
%
%
%
	\begin{equation}
	t = \frac{-C^2}{\beta \cc C - C^2}>1.
	\end{equation}
	
It is an immediate computation to see that $\gamma \in \Pos(X)$	.
	
%
%
%
%
%

Hence we have $\gamma = t \beta + (1-t) C \in G= \Pos(X) \cap C^{\perp}$ and thus, $\beta = \frac{1}{t}\gamma + \frac{t-1}{t}C \in G + R(C)$.

	
	Now, using \eqref{claim:span} and \eqref{eq:GCKsL}, we get $\alpha \in C^{\leqslant 0}\cap \Pos(X) \subseteq G + R(C) \subseteq (K-sL)^{\leqslant 0}$, a contradiction since $\alpha \in (K-sL)^{>0}$.
\end{proof}

\begin{claim}\label{claimContConv}
$\Nef(X)_{(K - sL)^{\geqslant 0}} = \Pos(X)_{(K - s L)^{\geqslant 0}}$.
\end{claim}
\begin{proof}[Proof of the Claim]
At first, extending Claim \ref{claimNefPos}, we prove the following
\begin{equation}\label{eqGIGI}
\partial \Nef(X) \cap (K-sL)^{\geqslant 0}= \partial \Pos(X)\cap (K-sL)^{\geqslant 0}.
\end{equation}
Indeed, by Claim \ref{claimNefPos}, taking the closure, we get
$$
\cl(\partial \Nef(X) \cap (K-sL)^{> 0})= \cl(\partial \Pos(X)\cap (K-sL)^{> 0}).
$$
Let us recall that the boundary satisfies a sort of Leibniz formula for two closed subsets $C,D$ of a topological space: $\partial(C \cap D)= (\partial C \cap D) \cup (C \cap \partial D)$. Hence, since $\intern (\partial \Pos(X))= \emptyset$, we see at once that $\cl(\partial \Pos(X) \cap (K-sL)^{>0}) = \partial \Pos(X) \cap (K-sL)^{\geqslant 0}$. Thus we get

%
%

$$
\begin{aligned}
& \partial \Pos(X) \cap (K-sL)^{\geqslant 0} = \cl(\partial \Nef(X) \cap (K-sL)^{>0}) \subseteq \\
 &\cl(\partial \Nef(X)) \cap \cl((K-sL)^{>0}) = \partial \Nef(X) \cap (K-sL)^{\geqslant 0}.
\end{aligned}
$$

To prove the other inclusion in \eqref{eqGIGI}, let us take $x \in \partial \Nef(X) \cap (K-sL)^{\geqslant 0}$. If it is in $(K-sL)^{>0}$, then it is in $\partial \Pos(X)$ by Claim \ref{claimNefPos}. Hence we can suppose $x \in (K-sL)^{\perp}$ and, by contradiction, $x \in \PosO(X)$; by the result of Campana and Peternell (see \eqref{camPet}), we have therefore that $x \in C^{\perp}$ for some $C$ with $C^2<0$ and thus $x \in C^{\perp} \cap (K-sL)^{\perp} \cap \PosO(X)$.

We can have two different cases. If $C^{\perp}= (K-sL)^{\perp}$ then $C$ and $(K-sL)$ have to be parallel, but since $C \cc h >0$ and $(K-sL) \cc h <0$, there must be an $a>0$ such that $aC = - (K-sL)$, which gives $0< -(K-sL) \cc C = a C^2 <0$, a contradiction.

If $C^{\perp} \neq (K-sL)^{\perp}$, since the origin and $x$ lie in both of them, they are not parallel and thus they intersect in a linear subspace of dimension $\rho(X)-2$. Now, since $x \in C^{\perp} \cap (K-sL)^{\perp} \cap \PosO(X)$, by dimension reasons, there will be an $y'\in C^{\perp}\cap \PosO(X) \cap (K-sL)^{>0}$, that is a contradiction with Claim \ref{claimG}.


%



Thus we have the \eqref{eqGIGI} and, by subtracting the equation in Claim \ref{claimNefPos}, we immediately see that
\begin{equation}\label{eq:ort}
\partial \Nef(X)\cap (K-sL)^{\perp} = \partial \Pos(X)\cap (K-sL)^{\perp}.
\end{equation}

Now we claim that
\begin{equation}\label{eqRIRI}
\Nef(X)\cap (K-sL)^{\perp} = \Pos(X)\cap (K-sL)^{\perp}.
\end{equation}
One of the two inclusion is obvious. To prove the other, let us take $x \in \Pos(X) \cap (K-sL)^{\perp}$; if $x \in \partial \Pos(X)$, then by equation \eqref{eq:ort}, we are done; if otherwise $x \in \PosO(X) \cap (K-sL)^{\perp}\subset (K-sL)^{\perp}$, it is in the convex hull of its boundary as a closed cone in $(K-sL)^{\perp}$ and we can write
$$
x = \sum \gamma_i, \quad \gamma_i \in \partial_{(K-sL)^{\perp}}(\Pos(X) \cap (K-sL)^{\perp}) \subseteq \partial \Pos(X) \cap (K-sL)^{\perp},
$$
where the last inclusion comes from the fact that if $C \subset W$ is a closed subset and $T \subset W$ is a topological subspace, then $\partial_H(C \cap H) \subseteq \partial C \cap H$.

Now equation \eqref{eq:ort} allows us to write $x = \sum \gamma_i$, with $\gamma_i \in \partial \Nef(X) \cap (K-sL)^{\perp}$; then $x \in \Nef(X)\cap (K-sL)^{\perp}$ and equation \eqref{eqRIRI} is proved. Hence by \eqref{eqGIGI} and \eqref{eqRIRI}, we get:

$$
\begin{aligned}
&\partial (\Nef(X) \cap (K-sL)^{\geqslant 0}) = (\partial \Nef(X) \cap (K-sL)^{\geqslant 0}) \cup (\Nef(X) \cap (K-sL)^{\perp})  \\
=\phantom{i}& (\partial \Pos(X) \cap (K-sL)^{\geqslant 0}) \cup (\Pos(X) \cap (K-sL)^{\perp}) =\partial (\Pos(X) \cap (K-sL)^{\geqslant 0}).
\end{aligned}
$$
Since we have two closed and convex cones not containing lines with the same boundary, their convex hull is the same and the claim is proved.
\end{proof}
We are now getting closer to the conclusion: our goal is a sort of \emph{dual statement} of Claim \ref{claimContConv}.

At first let us prove that
\begin{equation}\label{eq:neinter}
\NE(X)\cap (K-sL)^{\perp} = \Pos(X) \cap (K-sL)^{\perp}.
\end{equation}
Since $\Pos(X)\subseteq \NE(X)$, one of the two inclusion is obvious. For the other inclusion, let us suppose, by contradiction that there exists $\gamma \in \NE(X) \cap (K -sL)^{\perp}$ with $\gamma^2 <0$.

If we consider the rays outgoing from $\gamma$ and tangent to $\Pos(X)$, we see that, since $(K-sL)^2<0$ (see the proof of Proposition \ref{(-n,p)prop}), by Lemma \ref{lemmaConoMio}, there are rays in both $(K-sL)^{<0}$ and $(K-sL)^{>0}$ side. Thus we can fix two tangent rays intersecting $\partial \Pos(X)$ in $\alpha$ and $\beta$ such that:
\begin{equation}\label{eq:alfaBeta}
\alpha, \beta \in \gamma^{\perp}; \quad \alpha^2= \beta^2=0; \quad \alpha \in (K-sL)^{>0}; \quad \beta \in (K-sL)^{<0}.
\end{equation}


We point out that since $\alpha \in (K-sL)^{>0}$ and $\beta \in (K-sL)^{<0}$, then $\alpha$ and $\beta$ are not proportional and thus the segment $[\alpha,\beta]$ can't be contained in $\partial \Pos(X)$ and therefore the open segment $(\alpha, \beta)$ does lie in $\PosO(X)$ (see the proof of Fact \ref{iperpianiC^2>0}). 

Intersecting the segment $(\alpha, \beta)$ with $(K-sL)^{\perp}$, we found $y \in (\alpha, \beta) \cap \PosO(X)$ corresponding to a certain $\bar t \in (0,1)$. Since $\alpha, \beta \in \gamma^{\perp}$, we get at once: $y \in \gamma^{\perp} \cap (K-sL)^{\perp} \cap \PosO(X)$.

%

Now $y$ is in the interior of $\Pos(X)$ and $\gamma$ in the exterior, hence there is an $x \in (y, \gamma)$ such that $x \in \partial \Pos(X)$, that is $x^2=0$.

We immediately see that $x \cc (K-sL) = t y \cc (K-sL) + (1-t) \gamma \cc (K-sL) =0$, and hence $x \in \Pos(X)_{(K-sL)^{\geqslant 0}}$.

On the other side, if we compute $x \cc \gamma = t y \cc \gamma + (1-t) \gamma^2 <0$, we see that, since $\gamma \in \NE(X)$, then $x$ can't be a nef class and this is a contradiction with Claim \ref{claimContConv}.

We want now finally prove that
$$
\NE(X)_{(K-sL)^{\geqslant 0}}=\Pos(X)_{(K-s L)^{\geqslant 0}}.
$$
Since $\Pos(X) \subseteq \NE(X)$ we have that one of the two inclusions is obvious. In order to prove the other, suppose, by contradiction, that there exists $x \in \NE(X)\cap (K-sL)^{\geqslant 0}$ such that $x \notin \Pos(X)$.

By an argument of extremal rays, we can suppose that $x = [C]$ for some integral curve with $C^2<0$.

%

%
%
%
%

Now, as in Claim \ref{claimG}, setting $G = C^{\perp}\cap \PosO(X)$, we get $G \subset (K-sL)^{\leqslant 0}$.



Let us fix a $\gamma \in G\neq \emptyset$; since $\gamma \cc(K-sL) \leqslant 0$ and $C \cc (K-sL)\geqslant 0$, the segment joining $C$ to $\gamma$ does intersect $(K-sL)^{\perp}$: the line $L(C,\gamma)=\{\lambda(t)= t\gamma + (1-t)C \mid t \in \RR\}$, intersects $(K-sL)^{\perp}$ in $\bar{\lambda}=\lambda(\bar t)$ for some $0 < \bar t \leqslant 1$. It is easy to see that $\bar \lambda \cc C \leqslant 0$ and that $\bar{\lambda} \in \PosO(X)$.



Let us set $\lambda_{\varepsilon} = \lambda(\bar t -\varepsilon)$, for some $0< \varepsilon \ll 1$; an immediate computation shows that $\lambda_{\varepsilon} \in (K-sL)^{\geqslant 0}$.

%
Now, since $\varepsilon$ is small, we have that $\lambda_{\varepsilon} \in \Pos(X) \cap (K-sL)^{\geqslant 0} = \Nef(X) \cap (K-sL)^{\geqslant 0}$; in particular $\lambda_{\varepsilon}$ is nef; on the other side, we immediately get
$$
C \cc \lambda_{\varepsilon} = [(\bar t-\varepsilon)\gamma + (1-\bar t + \varepsilon)C]\cc C = (\bar t-\varepsilon) \underbrace{\gamma \cc C}_{=0} + \underbrace{(1-\bar t +\varepsilon)}_{>0} \underbrace{C^2}_{<0}<0,
$$
that is a contradiction. 
\end{proof}

\section{Strict inclusion conditions}\label{sec:strict}

We have now seen that, assuming some conjecture, if we blow up a sufficiently large number of points, then the Mori cone $\NE(X)$ does coincide with the positive cone in the $(K-sL)^{\geqslant 0}$ part. Our goal is now to show that, independently of any conjecture, the restriction of the positive cone to $K^{\geqslant 0}$ can't coincide with the restriction of $\NE(X)$. 

%

%

\begin{propo}\label{propoAlfa}
Let $X= \Bl_r Y$ be the blow up at $r$ general points of a smooth projective surface $Y$ and $A$ be an ample divisor. Let us suppose one of the following holds true.
$$
(A) \quad
\begin{cases}
r \leqslant {K_Y}^2+1 - \frac{(A\cc K_Y)^2}{A^2}\\
A \cc K_Y >0 \\
A^2<(A \cc K_Y)^2;
\end{cases}
\quad
(B) \quad
\begin{cases}
r > {K_Y}^2+1 - \frac{(A\cc K_Y)^2}{A^2}\\
r \leqslant {K_Y}^2+1\\
A \cc K_Y >0
%
%
%
\end{cases}
$$

$$
(C) \quad
\begin{cases}
r>0 \\
{K_Y}^2<0;
\end{cases}
\quad
(D)\quad
\begin{cases}
r>{K_Y}^2+1\\
{K_Y}^2 \geqslant 0;
\end{cases}
$$
%

%

Then, for a fixed $(-1)$-curve $C$, there exists $\alpha \in \Pos(X)$ such that
\begin{equation}\label{condalfa}
\begin{cases}
\alpha^2=0, &\alpha \cc h \geqslant 0 \\
\alpha \cc C \leqslant 0,& \alpha \cc K >0.
\end{cases}
\end{equation}

Moreover, we get: $\Pos(X)_{K^{\geqslant 0}} \subsetneq \NE(X)_{K^{\geqslant 0}}$.
%
\end{propo}

\begin{proof}
Let us fix $C=E_i$ for some $i$, one of the exceptional curves. At first we prove that conditions \eqref{condalfa} give the strict inclusion. Let us set $\gamma = C + \lambda \alpha$, with $\lambda \gg 1$. Since $\alpha \in \Pos(X)\subseteq \NE(X)$, then $\gamma \in\NE(X)$; on the other side $\gamma^2 = (C + \lambda \alpha)^2= C^2 + 2 \lambda C \cc \alpha <0$, which gives $\gamma \notin \Pos(X)$. Now, since $\lambda \gg 1$ and $\alpha \cc K >0$, we get $(C + \lambda \alpha)\cc K = -1 + \lambda \alpha \cc K >0$; hence $\gamma \in \NE(X)_{K^{>0}}$ and $\gamma \notin \Pos(X)_{K^{>0}}$.

We now look for an $\alpha$ in the form $\alpha =t C - (K-sL)$, with $t,s \in \RR$; we want to show the existence of $t,s$ in order to fulfil conditions \eqref{condalfa}. First of all, we need
\begin{equation}\label{eqTC}
\alpha^2 = (tC-(K-sL))^2=0.
\end{equation}
To ensure the existence of solutions for $t$ of \eqref{eqTC}, we require $\Delta_t := (C \cc (K-sL))^2+(K-sL)^2 \geqslant 0$, that, by adjunction and by Fact \ref{conticini}, becomes
\begin{equation}\label{eqS}
s^2A^2 - 2sA \cc K_Y +{K_Y}^2+1-r \geqslant 0.
\end{equation}
Thus, according to the sign of the discriminant $\Delta_s$ of equation \eqref{eqS}, we have two different cases:

\begin{equation}\label{S1}
\text{Case $\Delta_s \geqslant 0$} \quad \begin{cases}
r \geqslant {K_Y}^2+1-\frac{(A \cc K_Y)^2}{A^2} \\
s \leqslant \frac{A \cc K_Y - \sqrt{\Delta_s}}{A^2} \quad \vee \quad s\geqslant \frac{A \cc K_Y + \sqrt{\Delta_s}}{A^2}.
\end{cases}
\end{equation}
\begin{equation}\label{S2}
\text{Case $\Delta_s <0$} \quad \begin{cases}
r < {K_Y}^2+1-\frac{(A \cc K_Y)^2}{A^2} \\
\forall s \in \RR.
\end{cases}
\end{equation}
With this conditions on $s$, $\Delta_t \geqslant 0$ and, among the solutions of \eqref{eqTC}, we pick $t = 1 + \sqrt{\Delta_t}$.

We now impose $\alpha \cc h\geqslant 0$, for an ample class $h= L-\sum \delta_j E_j$. An easy computation, since $0 <\delta_j\ll 1$, shows that
$\alpha \cc h \geqslant 0$ if and only if $(sA^2-A \cc K_Y) >0$, that is

\begin{equation}\label{cond4}
s > \frac{A\cc K_Y}{A^2}.
\end{equation}
Now, the case $\Delta_s =0$ in \eqref{S1} can be associated to equation \eqref{S2} and these two conditions, together with \eqref{cond4}, become
\begin{equation}
\begin{cases}
r \leqslant {K_Y}^2+1-\frac{(A \cc K_Y)^2}{A^2} \\
s > \frac{A\cc K_Y}{A^2},
\end{cases}
\quad \text{and}\quad
\begin{cases}
r > {K_Y}^2+1-\frac{(A \cc K_Y)^2}{A^2} \\
s \geqslant \frac{A\cc K_Y+\sqrt{\Delta_s}}{A^2}.
\end{cases}
\end{equation}

We see at once that, since $t\geqslant 1$, then $\alpha \cc C=1-t \leqslant 0$. To prove \eqref{condalfa} it is left to deal with $\alpha \cc K$; since $\alpha \cc K = -t -{K_Y}^2+r+sA \cc K_Y$, the condition to impose is
%
%

\begin{equation}\label{eqKK}
r-{K_Y}^2-1+s A \cc K_Y > \sqrt{\Delta_t}.
\end{equation}
At the end, we get two different systems of inequalities for $s$:

\begin{equation}\label{eqSIS1}
\begin{cases}
r \leqslant {K_Y}^2+1-\frac{(A \cc K_Y)^2}{A^2} \\
s > \frac{A\cc K_Y}{A^2}\\
r-{K_Y}^2-1+s A \cc K_Y > \sqrt{\Delta_t},
\end{cases}
\quad \text{and}\quad
\begin{cases}
r > {K_Y}^2+1-\frac{(A \cc K_Y)^2}{A^2} \\
s \geqslant \frac{A\cc K_Y+\sqrt{\Delta_s}}{A^2}\\
r-{K_Y}^2-1+s A \cc K_Y > \sqrt{\Delta_t}.
\end{cases}
\end{equation}
The hypothesis in the statement of the proposition are exactly the conditions ensuring the existence of solutions for $s$ in \eqref{eqSIS1}. To solve \eqref{eqSIS1}, we used the computational system Wolfram Alpha (\texttt{http://www.wolframalpha.com/}). Setting $x=A \cc K_Y, y= A^2, z={K_Y}^2+1$, the solutions of \eqref{eqSIS1} are given by the strings:
\begin{verbatim}
Reduce[{r <= -(x^2/y) + z, r > 0, y > 0, s > x/y, 
r + s x - z > Sqrt[-r - 2 s x + s^2 y + z]}, s]

Reduce[{r > -(x^2/y) + z, r > 0, y > 0,s >= (x + Sqrt[x^2 - y (-r + z)])/y, 
r + s x - z > Sqrt[-r - 2 s x + s^2 y + z]}, s]
\end{verbatim}
An easy refinement of the computed solution gives the result.
\end{proof}
%

We can now give a similar statement in the case of an interesting geometrical hypothesis.

\begin{propo}\label{prop:uni}
Let $X=\Bl_r Y$ the blow up at $r\geqslant 2$ general points of a projective surface $Y$; let us suppose that for an ample divisor $A$ on $Y$ the inequality
\begin{equation}\label{cond:AK}
A \cc K_Y + \sqrt{A^2(r-1)}>0
\end{equation}
holds true, then $\Pos(X)_{K^{\geqslant 0}} \subsetneq \NE(X)_{K^{\geqslant 0}}$. In particular this is true if $Y$ is a non uniruled surface.
\end{propo}
\begin{proof}
In light of Proposition \ref{propoAlfa}, we just have to show, for a fixed $(-1)$-curve, the existence of an $\alpha$ satisfying \eqref{condalfa}.

Let us fix $C = E_r$, the last exceptional curve on $X$, and let $A$ be an ample divisor $A$ on $Y$. We look for an $\alpha$ in the form
$$
\alpha = \fistar A + \sum_{i=1}^r a_i E_i \quad \text{ with }a_i \in \RR.
$$
Imposing $\alpha \cc C =0$ gives $a_r=0$ and hence we can write $\alpha = \fistar A + \sum_{i=1}^{r-1}a_i E_i$. The $\alpha^2=0$ condition gives
\begin{equation}\label{eqZ1}
\alpha^2 = A^2 - \sum_{i=1}^{r-1}{a_i}^2=0 \quad \Rightarrow \quad A^2 =  \sum_{i=1}^{r-1}{a_i}^2.
\end{equation}
This condition is satisfied, for example, setting
$$
a_i = - \sqrt{\frac{A^2}{r-1}}, \quad \text{ for }i=1, \ldots, r-1; \quad a_r=0.
$$
We can now compute $\alpha \cc h$ for an appropriate ample class $h= L-\sum \delta_i$, with $0< \delta_i \ll 1$; we have
$$
\alpha \cc h = \left( \fistar A - \sum_{i=1}^{r-1} \sqrt{\frac{A^2}{r-1}}E_i \right) \cc \left( \fistar A - \sum_{i=1}^{r}\delta_i E_i\right) =A^2 - \sum_{i=1}^{r-1}\sqrt{\frac{A^2}{r-1}}\delta_i,
$$

that is positive since $\delta_i\ll 1$ and $A^2>0$. At the end we have:
\begin{equation}\label{eqZ2}
\alpha\cc K =\fistar A \cc \fistar K_Y - \sum_{i=1}^{r-1}a_i = A \cc K_Y - \sum_{i=1}^{r-1}\left(-\sqrt{\frac{A^2}{r-1}}\right),
\end{equation}
that is positive by \eqref{cond:AK}. In the non uniruled case, we have in particular that $K_Y$ is a pseudoeffective divisor, hence $A \cc K_Y \geqslant 0$ and condition \eqref{cond:AK} is immediately satisfied.
\end{proof}

\begin{rem}
We have that in the case $Y= \PP^2$, Proposition \ref{propoAlfa} and Proposition \ref{prop:uni} give the same bound $r >10$. Thus if we blow up $r > 10$ points, then $\Pos(X)_{K^{\geqslant 0}} \subsetneq \NE(X)_{K^{\geqslant 0}}$ and we have recovered the same results of \cite{deF}.
\end{rem}

\vspace*{1cm}
\normalsize

\begin{flushright}
Dipartimento di Matematica ``Guido Castelnuovo''

Universit\`a di Roma ``La Sapienza''

Piazzale Aldo Moro 5, 00185 Roma, Italia

E-mail address: \texttt{disciullo@mat.uniroma1.it}

\texttt{fulviodisciullo@gmail.com}
\end{flushright}
\end{document}